\documentclass[a4paper,12pt]{article}
\usepackage{tikz}
\usetikzlibrary{calc}
\usepackage{amssymb}
\usetikzlibrary{quotes,angles}
\usetikzlibrary{plotmarks}
\usetikzlibrary{arrows,shapes,positioning}
\usetikzlibrary{decorations.markings}
\usepackage{geometry}
\tikzstyle arrowstyle=[scale=2]
\tikzstyle directed=[postaction={decorate,decoration={markings,
    mark=at position .65 with {\arrow[arrowstyle]{stealth}}}}]
\tikzstyle reverse directed=[postaction={decorate,decoration={markings,
    mark=at position .65 with {\arrowreversed[arrowstyle]{stealth};}}}]
    \tikzstyle left directed=[postaction={decorate,decoration={markings,
    mark=at position -.62 with {\arrow[arrowstyle]{stealth}}}}]

\tikzstyle left reverse directed=[postaction={decorate,decoration={markings,
    mark=at position -.62 with {\arrowreversed[arrowstyle]{stealth};}}}]
\usepackage{indentfirst}
\usepackage[plainpages=false]{hyperref}
\usepackage{amsfonts,latexsym,rawfonts,amsmath,amssymb,amsthm,mathrsfs}
\usepackage{amsmath,amssymb,amsfonts,latexsym,lscape,rawfonts}

\usepackage[all]{xy}
\usepackage{eufrak}
\usepackage{makeidx}         
\usepackage{graphicx,psfrag}
\usepackage{array,tabularx}

\usepackage{setspace}
\usepackage{tocloft}

\newtheorem{thm}{Theorem}[section]

\newtheorem{lem}[thm]{Lemma}

\newtheorem{prop}[thm]{Proposition}

\theoremstyle{remark}
\newtheorem{rmk}[thm]{Remark}
\newtheorem{fact}[thm]{\textbf{Fact}}

\theoremstyle{definition}
\newtheorem{Def}[thm]{Definition}          
      
\newtheorem{Notation}[thm]{Notation}              
\newtheorem{Terminology}[thm]{Terminology}

\title{Atiyah classes and the essential obstructions in deforming a singular $G_{2}-$instanton}

{\small{\author{Yuanqi Wang\thanks{Department of Mathematics, The University of Kansas, Lawrence, KS, USA.\ yqwang@ku.edu.}}}

\date{\vspace{-5ex}}

\begin{document}
\maketitle
\begin{abstract}When the rank of the bundle is $\geq 2$, in a certain sense, we found an essential obstruction  for the gluing construction of $G_{2}-$instantons with $1-$dimensional singularities. It involves the Atiyah classes  generated by  contracting a vector in $\mathbb{C}^{3}$ with the curvature. Intuitively speaking,  the gluing  does not work if  the tangent connection at
a component of the  $1-$dimensional singular locus is not the twisted Fubini-Study connection on a twisted tangent bundle of $\mathbb{P}^{2}$. Particularly, it fails if the rank of the  bundle is $\geq 3$.  

\end{abstract}

\addtocontents{toc}{\protect\setcounter{tocdepth}{1}}
\section{Introduction}


Gauge theory  plays an important role in the differential topology of $4-$manifolds.  Corresponding  to  the  groups $SU(3)$, 
$G_{2}$, and $Spin(7)$ in the holonomy  list of Berger-Simons \cite{Berger}, \cite{Simons},  Donaldson-Thomas \cite{DT}  and  Donaldson-Segal \cite{DS}  intend to generalize the  gauge theory in dimensions $2,\ 3,\ 4$ to $6,\ 7,\ 8$.  In  dimension $7$, the connections of interest are the projective $G_{2}-$instantons i.e. a $U(n)-$connection $A$ such that the  curvature $F^{0}_{A}$ of the induced $PU(n)-$connection  satisfies the following equation 
\begin{equation}\label{equ proj instanton}
\star(F^{0}_{A}\wedge \psi)=0,
\end{equation}
where $\psi$ is the co-associative $G_{2}-$form. To  understand the boundary of the  moduli and  to construct examples of singular instantons  via gluing, a Fredholm  theory is important. For instantons with isolated singularities, the indicial roots are discrete. 
However,  those of $1-$dimensional singularities are not.  Finite dimensional obstructions can prevent a gluing construction: see the work of Brendle Kapouleas  \cite{BK} on Einstein metrics. Infinite dimensional obstructions make it even harder: see the work of Chen  \cite{GaoChen} on twisted connected sum of  $G_{2}-$structures with conical singularities along circles.  An option is to  add parameters into the domain Banach space. For  $\infty-$dimensional co-kernel, we need an $\infty-$dimensional  parameter space. 
On singular $G_{2}-$instantons and Hermitian Yang-Mills connections, in addition to deforming the connection,  we can pull back the $G_{2}-$structures by certain diffeomorphisms in which the Frechet partial derivative yields  the \textit{Auxiliary operator}. This yields the extended linearized operator. Our main result shows  a necessary condition for such a scheme to work for singular $G_{2}-$instantons.

 \textbf{Theorem A}. \textit{In an ideal  configuration  of $G_{2}-$instanton with $1-$dimensional singularities (Definition \ref{Def gdc}),
  the usual linearized operator \eqref{equ introduction formula for deformation operator} does not have closed range}.   \textit{The extended linearized operator \eqref{equ l tilde} has closed range only if at each circle $\gamma_{i}$ (as in Definition \ref{Def gdc}), the model connection  is  the twisted Fubini-Study connection on a twisted  tangent bundle of $\mathbb{P}^{2}$. Particularly,  the rank of the bundle must be $2$}. \\

The twisted Fubini-Study connection is defined up to a smooth bundle  isomorphism on $\mathbb{P}^{2}$. For bounded linear operators between Banach spaces, the range is not closed implies $\infty-$dimensional co-kernel. Theorem \textbf{A} implies the non-vanishing of the co-kernel on compact $7-$folds, including the twisted connected sums \cite{SWal}.  The co-kernel is called the obstruction but  is different from the \textit{essential obstruction} below.\\
 

 \textbf{Corollary B}:  \textit{In the tame configuration  of $G_{2}-$instanton with $1-$dimensional singularities over a compact $7-$manifold $M$ under Definition \ref{Def gdc}, \eqref{equ usual L compact}, and \eqref{equ extended  usual L compact}, for any $\delta\geq 0$, 
  the usual linearized operator \eqref{equ usual L compact} is not surjective. Suppose 
  \begin{itemize}\item  the rank of the bundle is $2$ but the  model connection  at some circle  is  not the twisted Fubini-Study connection on a twisted  tangent bundle of $\mathbb{P}^{2}$, or
  \item   the rank of the bundle  is  $\geq 3$.
  \end{itemize}
 Then the extended linearized operator \eqref{equ extended  usual L compact} is not surjective}. \\

Gluing construction of $G_{2}-$instantons with $1-$dimensional singularities on twisted connected sums is mentioned by Jacob-Walpuski in \cite{JW}. 
Corollary \textbf{B} says that the gluing is obstructed if  one of the tangent connections is  not twisted Fubini-Study. This is different from  smooth $G_{2}-$instantons on twisted connected sums considered by S\'a Earp-Walpuski \cite{SWal}, in which assuming the two Lagrangian subspaces in a sheaf cohomology intersect transversally \cite[Theorem 1.2]{SWal}, the co-kernel is trivial \cite[Theorem 3.24 Step 2]{SWal}. This is indeed the case for the concrete examples \cite{Walpuski}, \cite{MNS}.

We do not  need the $G_{2}-$structure to be globally co-closed though  it might well be in the cases of interest. We only need the flexible functorial conditions I-V   which can be easily verified for the example in Corollary \textbf{B}. 

On other geometric objects,  there are perturbation  theories  deforming the singular locus.   For example, see Takahashi's deformation \cite{Takahashi} of $\mathbb{Z}_{2}-$harmonic spinors in dimension $3$. Very  recently, Donaldson \cite{Donaldson} developed the deformation for multi-valued harmonic functions. Similarly to  \cite{Donaldson}, here any Green's function must possess a leading term disabling the deformation. In a certain sense it can not be ``overcome" by adding the vector fields if the essential obstruction does not vanish (Lemma \ref{lem bad solution go to infty when k goes to infinity} below).   For minimal surfaces with non-isolated singularity, please see the work of Mazzeo-Smale \cite{MazzeoSmale} that perturbs the singularities away.  This generalizes  Hardt-Simon  perturbation \cite{HS} for isolated singularities. 
  

 In the model setting of Corollary \textbf{B}, through a natural linear injection, the contraction of a (constant) vector in $\mathbb{C}^{3}$ with the curvature $F_{A}^{0}$ is mapped to the cohomology $H^{1}[\mathbb{P}^{2},(EndE)(-1)]$. This  is called an Atiyah class. They form a complex  $3-$dimensional subspace.  We define the \textit{essential obstruction} as the finite dimensional quotient  \begin{equation}\label{equ def ess obstruction}\frac{H^{1}[\mathbb{P}^{2},(EndE)(-1)]}{\{\textrm{Atiyah classes}\}}.\end{equation}
 Please see Proposition  \ref{prop atiyah} below. On gluing construction of Einstein metrics, Biquard \cite{Biquard} also found an obstruction involving curvature. 

Theorem \textbf{A} can be understood as a ``good news" for the compactification of moduli of smooth instantons, in conjunction with the work of Tian \cite{T} and the codimensional $6$ conjecture therein. On a compact  $G_{2}-$manifold with a unitary vector bundle,  it is reasonable to ask ``how often" (in a sense that needs to be specified) we  can see other model connections than the twisted Fubini-Study as the singularity model of  the ``limit" of a sequence of smooth ones. 

The vector fields we allow are spanned by all the $7$ directions (coordinate vectors) near a component of a singular circle while the coefficients only depend on $r$ and $s$. Please see Section \ref{subsect vect fields} below. This is  the advantage of the Euclidean space $\mathbb{C}^{3}\setminus O$ as the simplest Calabi-Yau cone: there are constant vector fields on the $7-$dimensional product $(\mathbb{C}^{3}\setminus O)\times \mathbb{S}^{1}$ deforming the circle $O\times \mathbb{S}^{1}$ and also generating eigen-sections of the link operator with respect to eigenvalue $-1$. The case of more general vector fields remains mysterious.  We do not know whether there is any analogous structure for general Calabi-Yau cone over a regular Sasakian Einstein $5-$manifold. 

Briefly  speaking, under the conditions therein, the proof of  Theorem \textbf{A} for the extended linearized operator  is an assembling of the following 3 facts.
\begin{enumerate}\item The essential obstruction vanishes if and only if the tangent connection is a twisted Fubini-Study connection on $T^{1,0}\mathbb{P}^{2}(k)$ (Lemma \ref{cor rigidity of bundle on P2}).
\item  The range of the model auxiliary operator is in the span of Atiyah classes (Proposition \ref{lem equ for X vs equ for u and v}).
\item Under the surjectivity condition III$^\star$, if the essential obstruction is non-trivial, we can construct a singular sequence violating closed range (Lemma \ref{lem bad solution go to infty when k goes to infinity}). 
\end{enumerate}

Organization of the paper:  Almost all definitions related to Theorem \textbf{A} and Corollary \textbf{B} are in Sections \ref{sect Preliminary} and \ref{sect ext lin}. In Section \ref{section local theory} we define the Atiyah classes in $Eigen_{-1}P$ and use Riemann-Roch to show that the cohomology $H^{1}[\mathbb{P}^{2},(EndE)(-1)]$ consists only of Atiyah classes is equivalent to that $E$ is a twisted tangent bundle. In Section \ref{Sect Aux} we state and prove the formula for the auxiliary operator, leaving  routine tensor calculations to the Appendix. In Section \ref{proof of Theorem A} we prove the main results using separation of variables, Sasakian geometry of the linearized operator, modified Bessel functions, and functional analysis. \\

\textbf{Acknowledgement:} The author is grateful  to Simon Donaldson for  encouragements and  conversations on this project. 

\section{Preliminary\label{sect Preliminary}}
In this section we define the configuration required in Theorem \textbf{A}. 
 \begin{Def}\label{Def gdc}  Throughout, a ball $B(R)$ is always in $\mathbb{C}^{3}$ and centred at the origin.  A  \textit{tame configuration of $G_{2}-$instanton with $1-$dimensional singularities} consists of: 
 \begin{enumerate}
 \item finitely-many disjoint embedded circles (embedded $\mathbb{S}^{1}$'s) $\gamma_{i},\ i=1,...,l$ with trivial normal bundle in a $7-$manifold $M$, and mutually disjoint  tubular neighborhood of  $\gamma_{i}$ diffeomorphic to $[B(100R_{0})\setminus O]\times \mathbb{S}^{1}$ for some $R_{0}>0$;
 \item a smooth unitary connection $A$ on a bundle $E\rightarrow M\setminus \gamma$ with rank  $n\geq 2$ such that  in each tubular ball as above, $(A,E)$ is  equal to the pullback of a non-projectively flat Hermitian Yang-Mills connection $(A_{i},  E_{i})\rightarrow \mathbb{P}^{2}$ via the standard fibration map $(\mathbb{C}^{3}\setminus O)\times \mathbb{S}^{1}\rightarrow \mathbb{P}^{2}$;
 \item a $G_{2}-$structure on $M$  equal to the standard one near each $\gamma_{i}$ under the same coordinate;
 \item Banach spaces $Y$, $\mathcal{B}$, and $\chi(M,TM)$ that satisfy  condition I, IV, and V below. 
 \end{enumerate}

A tame configuration is \textit{ideal}  if condition II, III, and III$^{\star}$ hold. 
\end{Def}
The reason we can assume $R_{0}$ is  independent of $i$ is that there are only finitely-many circles. 
The results in the introduction are independent of $R_{0}$ as long as it is $>0$. Many discussions below are under the coordinate chart in the first bullet point above. This should be  clear from context. For example, see condition II below. 

The following terms make it convenient. 
\begin{Terminology}\label{Term punched ball}The manifold $(\mathbb{C}^{3}\setminus O)\times \mathbb{S}^{1}$ is called the \textit{model space}. The open set  $B(R)\times \mathbb{S}^{1}$  and the punched set  $[B(R)\setminus O]\times \mathbb{S}^{1}$ are called the \textit{tubular ball} and \textit{punched tubular ball}, respectively. The punched tubular ball with  radius $R=\infty$ is the model space. 
\end{Terminology}

Let $r$ denote the distance to the origin in $\mathbb{C}^{3}$. This is also the Euclidean distance to the circle $O\times \mathbb{S}^{1}$ in $\mathbb{C}^{3}\times \mathbb{S}^{1}$. Sometimes it is denoted by $r_{x}$ as a function (see \eqref{equ holder norm} below).  
 \subsection{The usual linearization}
Let $\Omega^{k}_{adE}$ denote the bundle of  $adE-$valued $k-$forms. With gauge fixing and monopole term, the usual linearization of   \eqref{equ proj instanton} in the connection is a first order elliptic operator $\underline{L}$ that maps $C^{\infty}[M^{7}\setminus\gamma, \Omega^{0}_{adE}\oplus \Omega^{1}_{adE}]$ to itself:
\begin{equation}\label{equ introduction formula for deformation operator}\underline{L}[\begin{array}{c}
\sigma   \\
  a   
\end{array}]=[\begin{array}{c}
d_{A}^{\star}a   \\
  d_{A}\sigma+\star(d_{A}a\wedge \psi)  
\end{array}], \end{equation}
 where $\sigma$ is a section and $a$ is a $1-$form, both $adE-$valued. To avoid heavy notation \textit{we henceforth suppress the bundles and even the domain manifold in the notation for the Banach spaces, including  the weighted Schauder spaces etc}. 

  Let the domain of the usual linearized operator be a Banach space $Y$ that is  a subspace of  $C^{1}(M\setminus \gamma)$. Likewise, let the target $\mathcal{B}$ be  a Banach space that is  a subspace of $C^{0}(M\setminus \gamma)$, such that the following holds. 
 
\begin{equation}\label{equ bound usual L}\textrm{Condition I}:\ \ \ \ \underline{L}:\ \ Y\rightarrow \mathcal{B}\ \ \textrm{is bounded}.\nonumber\end{equation}

To construct singular sequence, we need two more conditions. The first is the lower bound comparing the norm of $Y$ to the standard weighted $C^{0}-$norm whose sections are $O(\frac{1}{r})$ near the circles. 
 \begin{equation}\label{equ condition 1 Y}
\textrm{Condition II}:\ \ \ \  \|\xi\|_{Y}\geq N\|Res|_{\underline{R}_{0}}\xi \|_{C^{0}_{1}[B(\underline{R}_{0})\times \mathbb{S}^{1}]}\ \textrm{for some}\ 0<\underline{R}_{0}<R_{0}.\nonumber
 \end{equation}
  where $Res|_{\underline{R}_{0}}$ is the restriction of  $\xi$ onto the punched ball of radius $\underline{R}_{0}$. 

 The other condition is an upper bound on the $\mathcal{B}-$norms of  a particular sequence of compactly supported sections. Namely,  let $\chi$ be a cutoff function as below \eqref{equ bad h}.  We assume  there is a unit vector $\zeta\in Eigen_{-1}P$ 
(which is required to be perpendicular to the Atiyah classes if the essential obstruction is non-trivial) 
 such that
 \begin{equation}\label{equ condition 2 B}\textrm{Condition III}:\ \ \ \ 
 || \frac{y^{\delta}\chi(ky)K_{0}(ky)\sin ks}{k}\cdot I\zeta||_{\mathcal{B}}\leq C_{\mathcal{B},k_{0}}\  \textrm{for some}\ \delta\geq 0\nonumber
 \end{equation}
where  $C_{\mathcal{B},k_{0}}$ is a  constant  independent of integer $k\geq k_{0}$ for some  $k_{0}\geq 1$.  Moreover, we define 
$$ \textrm{Condition III}^{\star}:\ y^{\delta}\chi(ky)K_{0}(ky)\sin ks\cdot I\zeta  \in Range \underline{L}\ \ \textrm{for any}\ k \ \textrm{as above}. $$
The range of the extended linearized operator $L$ contains the range of the usual $\underline{L}$.
Because of the the exponential decay of the modified Bessel function of second kind $K_{0}(x)$ for  large $x\geq 1$ (see \cite{Watson} for a comprehensive theory), we expect no difficulty in checking condition III for a specific Banach space $\mathcal{B}$. Please see the proof of Corollary \textbf{B} below.



 \subsection{The vector fields\label{subsect vect fields}}
 Let $(e_{i},\ 1\leq i\leq 6)$ be the standard basis of $\mathbb{R}^{6}$ and $e^{i}$ be the dual basis. Near the circle $O\times \mathbb{S}^{1}\subset \mathbb{C}^{3}\times \mathbb{S}^{1}$, we consider vector fields of the following form. 
 \begin{equation}\label{equ vector field} X=X_{s}\frac{\partial}{\partial s}+\Sigma_{i=1}^{6}X_{i}e_{i}\ \textrm{where the coefficients}\ X_{s},\ X_{i}\ \textrm{only depend  on}\ r\ \textrm{and}\ s. 
 \end{equation}
The global  vector fields are as follows. 
\begin{Def}\label{Def vect fields}Let $\mathfrak{X}(M,TM)$ be a Banach space  of  vector fields on $M$ which is a subspace in  $C^{1}(M\setminus \gamma)$. We say it satisfies \textit{Condition} IV if  the restriction of an arbitrary vector field in $\mathfrak{X}(M,TM)$ onto the tubular ball $B(R_{0})\times \mathbb{S}^{1}$ of each  circle defines  a bounded linear map from $\mathfrak{X}(M,TM)$ to the space $\mathfrak{X}_{R_{0}}$ of vector fields of the form   \eqref{equ vector field} (across the circle) with norm
\begin{eqnarray*}
||X||_{\mathfrak{X}_{R_{0}}}&=&\Sigma_{j=1}^{l}\Sigma_{i=1}^{7}\{|X_{i}|_{C^{0,1}[B(R_{0})\times \mathbb{S}^{1}]}+|\frac{\partial X_{i}}{\partial r}|_{C^{0}[(B(R_{0})\setminus O)\times \mathbb{S}^{1}]}+|\frac{\partial X_{i}}{\partial s}|_{C^{0}[(B(R_{0})\setminus O)\times \mathbb{S}^{1}]}\end{eqnarray*} 
where $X_{7}\triangleq X_{s}$. We want our vector fields to be Lipschitz even across the circles in line with the existence and  uniqueness of flows. 
\end{Def}
  \section{The extended linearized operator\label{sect ext lin}}
  \subsection{The auxiliary operator}
  We pullback the $G_{2}-$structure via a diffeomorphism $\chi$ integrated from a vector field $X\in \mathfrak{X}(M^{7},TM^{7})$ (at $t=1$). The instanton equation becomes
  \begin{equation}\label{equ instanton with diffeo}\star_{\chi^{\star}\phi}[F^{0}_{A}\wedge (\chi^{\star}\psi)]=0. \end{equation}
By Cartan formula, assuming $A$ is a projective instanton, the linearization in the diffeomorphism at $Id_{M}$ yields the \textit{Auxiliary operator} :
  \begin{equation}\label{equ aux non model}\star_{\phi}(F^{0}_{A}\wedge d[X\lrcorner\psi])+\star_{\phi}[F^{0}_{A}\wedge (X\lrcorner d\psi)].
\end{equation}
If $A$ is projectively flat, it vanishes. The second term vanishes in the punched tubular balls near each circle as the $G_{2}-$structure therein is standard. 

Under Definition \ref{Def vect fields} on the vector fields, we assume the following on the extended linearized operator. 

\begin{equation}\label{equ condition IV}\textrm{Condition V}:\ \ \ \ \ L:\ \ \mathfrak{X}(M^{7},TM^{7})\oplus Y\rightarrow \mathcal{B}\ \textrm{is bounded},\nonumber\end{equation}
where $L$ is  the linearization of \eqref{equ instanton with diffeo} with respect to the connection $A$ and the diffeomorphism $\chi$ (still with monopole term and gauge fixing): 

\begin{equation}\label{equ l tilde}L\left|\begin{array}{c}X\\
\sigma   \\
  a   
\end{array}\right |=\left | \begin{array}{c}
d_{A}^{\star}a   \\
  d_{A}\sigma+\star(d_{A}a\wedge \psi)+\star [F^{0}_{A}\wedge d(X\lrcorner \psi)]+\star [F^{0}_{A}\wedge (X\lrcorner d\psi)]  
\end{array}\right |.  \end{equation}
Please compare \eqref{equ l tilde} with  formula \eqref{equ introduction formula for deformation operator} of the usual linearization. 
The definitions involved in Theorem \textbf{A} are all established. 

\subsection{The model problem}
The model data on $(\mathbb{C}^{3}\setminus O)\times \mathbb{S}^{1}$  is the pullback of  a non-projectively flat Hermitian Yang-Mills connection $A$ on a bundle $E\rightarrow \mathbb{P}^{2}$ with rank $\geq 2$ and the standard  $G_{2}-$structure $(\phi_{euc},\ \psi_{euc})$.  Here we abused notation with the bundle ``$E$" on the manifold in Definition \ref{Def gdc}.2. 
The model usual linearized operator is
 \begin{equation}\label{equ underlineL0}
 \underline{L}_{0}=I\cdot [\frac{\partial}{\partial s}-\underline{T}(\frac{\partial}{\partial r}-\frac{P}{r})]
 \end{equation}
on the pullback of  the bundle  $$Dom=adE^{\oplus (4)}\oplus \Omega_{sba}^{1}(adE)\rightarrow \mathbb{S}^{5}$$ whose  rank is $8\times rank(adE)$. Moreover,
\begin{itemize}
\item  $\Omega_{sba}^{1}(adE)$ is the bundle of semi-basic $adE-$valued $1-$forms i.e. the pullback of   $\Omega^{1}(adE)\rightarrow \mathbb{P}^{2}$,\item  and $I,\ \underline{T}$ are isometries of  $Dom$. They anti-commute and generate a quaternion structure by $I\underline{T}=-K$. \end{itemize}
Please  see \cite[Lemma 5.3]{W}) for more. 

Let $\cong/\cong_{\mathbb{C}}$ mean  real$/$complex isomorphisms between two finite dimensional vector spaces. Part of the spectral theory for the link operator $P$ in \cite[Theorem A and D]{W} is  the following diagram of isomorphisms:
\begin{center}\begin{tikzpicture}
  \node at (9,3.3) {$Eigen_{-2}P$}; 
  \node at (9,1.6) { $H^{1}[\mathbb{P}^{2},EndE(-2)]$};

			\draw[-,semithick] (8.9, 2) -- (8.9, 3);

	  \node at (4,3.3) {$Eigen_{-1}P$};
  \node at (4,1.6) {$H^{1}[\mathbb{P}^{2},EndE(-1)]$};

			\draw[-,semithick] (4.2, 3) -- (4.2,2);
		\draw[-,semithick] (5.7, 1.6) -- (7.3,1.6);
		\draw[-,semithick] (5, 3.3) -- (8,3.3);
		 \node at (6.5, 3.8) {``Serre duality"};
		  \node at (6.5, 1) {Serre duality};
		   \node at (3.8, 2.5) {$\cong_{\mathbb{C}}$};
		      \node at (9.3, 2.5) {$\cong_{\mathbb{C}}$};
		\end{tikzpicture}\end{center}
The symbol ``$Eigen_{\mu}P$" means the eigen-space of $P$ of the eigen-value $\mu$.


The extended linearized operator \eqref{equ l tilde} becomes
 \begin{equation}\label{equ underlineL0 extended}
 L_{0}(\xi,X)=I\cdot [\frac{\partial}{\partial s}-\underline{T}(\frac{\partial}{\partial r}-\frac{P}{r})]+\star[F_{A}\wedge d(X\lrcorner\psi_{euc})].
 \end{equation}
\subsection{A brief remark about usual weighted Schauder spaces \label{sect weighted schauder}}
Following \cite{JasonLotay} and for Corollary \textbf{B}, we discuss  the standard weighted Schauder spaces of bundle sections. \subsubsection*{On a punched tubular ball with radius $R$ for  the bundle $Dom$}
 Let the H\"older semi-norm  be \begin{equation}\label{equ holder norm} [u]_{C^{0,\alpha}_{0}} \triangleq \sup_{x,y\in [(B(R)\setminus O)\times \mathbb{S}^{1}],\  O\times \mathbb{S}^{1}\cap\overline{xy}=\varnothing}[\min(r_{x},r_{y})]^{\alpha} \frac{|P_{\overrightarrow{xy}}[u(x)]-u(y)|}{d^{\alpha}(x,y)}\end{equation}
where
\begin{itemize}
\item $r_{x}$ is the distance from the $\mathbb{C}^{3}-$component of $x$ to the origin (see below Terminology  \ref{Term punched ball}),
\item $\overline{xy}$ is the shortest line segment (geodesic) joining $x,\ y$ and realizing the distance $d(x,y)$, and
\item $P_{\overrightarrow{xy}}$ is the parallel transport from $x$ to $y$ via the segment and  the connection in the tame configuration. 
\end{itemize}

 Let the norm  $|u|_{C^{0}_{0}}$ (and  $|u|_{C^{0}}$ which means the same) be simply $\sup_{x\in [(B(R)\setminus O)\times \mathbb{S}^{1}]}|u|(x)$. It only depends  on the bundle metric thus also applies  to a vector field.  

The $C^{1,\frac{1}{2}}_{0}-$norm is defined by 
\begin{equation}\label{equ def C1 1/2 0}|u|_{C^{1,\frac{1}{2}}_{0}}=|u|_{C^{0,\frac{1}{2}}_{0}}+\Sigma_{D=\frac{\partial}{\partial r},\frac{\partial}{\partial s},\frac{\nabla_{\mathbb{S}^{5},A}}{r}}|Du|_{C^{0,\frac{1}{2}}_{1}}.\end{equation}

When $R$ is finite, according to the principle \cite[(6.10)]{Trudinger}, the above norm treats  $O\times \mathbb{S}^{1}$ as boundary but not the other piece $\partial [B(R)]\times \mathbb{S}^{1}$. 

The weighted Schauder space $C^{k,\frac{1}{2}}_{p}$ is simply defined by the  multiplication with the factor $r^{p}$: $$|\xi|_{C^{k,\frac{1}{2}}_{p}}\triangleq |r^{p}\xi|_{C^{k,\frac{1}{2}}_{0}},\ k=0,\ 1.$$
For example, a section is in $C^{k,\frac{1}{2}}_{1}$ if and only if the multiplication by $r$ is in $C^{k,\frac{1}{2}}_{0}$. This implies the norm is   $O(\frac{1}{r})$ near the circle $O\times \mathbb{S}^{1}$.
\subsubsection*{Over a compact manifold}
Under a tame configuration over a compact manifold $M$, we can finitely cover the whole manifold by 
\begin{itemize} \item tubular balls with radius $10R_{0}$ and 
\item  geodesic convex balls  away from the tubular balls of radius $7R_{0}$ centered at components of $\gamma$,  such that balls of double radius are still geodesic convex and avoid the same tubular balls.
\end{itemize}
This can be achieved by taking a small enough ball (regarding $R_{0}$ and the Riemannian metric induced by the $G_{2}-$structure) at  any point not in the  tubular balls of radius $10R_{0}$ (which some of the geodesic convex balls still intersect).  Therefore with the tubular balls of radius $10R_{0}$, an (open) cover is obtained.  Then take a finite sub-cover.

 The next step is to  \textit{simply use partition of unity to patch 
the local norms to get the global}. On the geodesics balls, the usual Schauder norm is defined as a special case in \cite[Definition 4.3]{JasonLotay}. According to our choice, the cutoff functions $\psi_{i}$ corresponding to each tubular ball (in the partition of unity)  is  $\equiv 1$ in the even smaller tubular balls of radius $5R_{0}$.

\section{Atiyah classes \label{section local theory}}
In this section we  show that the essential obstruction vanishes for any only for $T^{1,0}\mathbb{P}^{2}(k)$. 

We recall some Sasakian geometry on the standard round $\mathbb{S}^{5}$ (of radius $1$ in $\mathbb{R}^{6})$.  Let $\upsilon$ and $\eta$ be the standard Reeb vector field and contact form on $\mathbb{S}^{5}$. There are three forms $\frac{d\eta}{2}$, $H$, $G$ on $\wedge^{2}D^{\star}$, where $D^{\star}\triangleq \eta^{\perp}$ is the contact co-distribution of rank $4$. The metric contraction between a semi-basic ($D^{\star}-$valued)  $1-$form with each of the forms is a complex structure on $D^{\star}$ denoted by $J_{0},\ J_{H},\ J_{G}$ respectively. By  metric pulling down, these complex structures also act on the contact distribution $D\triangleq \upsilon^{\perp}$. They form a quaternion structure on both $D$ and $D^{\star}$. This structure can also be generalized to the bundle $Dom$ for the linearized operator. The action of $I$ on semi-basic $1-$forms (including $Eigen_{-1}P$ and $Eigen_{-2}$), is $J_{0}$, and  the action of $\underline{T}$ on these forms is $-J_{G}$. The quaternion structure is determined by  $$J_{H}J_{0}=J_{G}.$$
  Let $\star$ denote the Hodge star of the Euclidean metric on the model space $(\mathbb{C}^{3}\setminus O)\times \mathbb{S}^{1}$, and $\star_{D^{\star}}$ the one on the contact co-distribution with respect to the standard  metric on $\mathbb{S}^{5}$.

The pullback of the projective curvature $F^{0}_{A}$ of the Hermitian Yang-Mills connection on $\mathbb{P}^{2}$ is $\star_{D^{\star}}-$anti self-dual i.e.   invariant under the quaternion structure $J_{0},\ J_{H},\ J_{G}$.  Let $d_{0}$ denote the transverse exterior differential operator $d-\eta\wedge L_{\upsilon}$. The square $d^{2}_{0}$ does not vanish in general.   We call a $D-$valued vector semi-basic and let $\sharp_{D^{\star}}$ denote the  metric pulling up of a semi-basic vector (field), which is a semi-basic form.  Please see \cite[Section 3]{W} for a comprehensive discussion.

\subsection{The map}
Now we define the injection. 
\begin{prop}\label{prop atiyah}Let  $(E,A)\rightarrow \mathbb{P}^{2}$ be a non-projectively flat Hermitian Yang-Mills bundle  with rank $n\geq 2$. For any (constant) vector $Y\in \mathbb{R}^{6}$, the bundle valued $1-$form  $r(Y\lrcorner F^{0}_{A})$ lies in $Eigen_{-1}P$. The resultant linear map $$\boxminus: \mathbb{R}^{6}\rightarrow Eigen_{-1}P\ \ (\cong _{\mathbb{C}}H^{1}[\mathbb{P}^{2},(EndE)(-1)])$$  is a complex injection. It is an isomorphism if and only if $E=(T^{1,0}\mathbb{P}^{2})(k)$. 
\end{prop}
A cohomology class in  $Range \boxminus\subset H^{1}[\mathbb{P}^{2},(EndE)(-1)]$,  in view of \cite{Atiyah}, is called an Atiyah class. The same term applies to an element in $Range \boxminus\subset Eigen_{-1}P$ via the complex isomorphism in \cite[Theorem A and Proposition 8.2]{W}.
\begin{Notation}\label{Notation Atiyah class} Denote $Range\boxminus$ by 
$$\{\textrm{Atiyah Classes}\}|_{Eigen_{-1}P}.$$
Suppressing the subscript,  this is the space on the ``denominator" in \eqref{equ def ess obstruction}. 
\end{Notation}
We need two facts for Proposition \ref{prop atiyah}. 
\begin{lem}\label{equ closed contraction}  Let $F^{0}_{A}$ be the curvature of the projective connection induced by a Hermitian connection 
over a K\"ahler surface. Suppose $F^{0}_{A}$ is $(1,1)$. 
\begin{itemize}\item If $X$ is a holomorphic $(1,0)-$vector field, then $X\lrcorner F^{0}_{A}$ is $\overline{\partial}_{A}-$closed.
\item If $X$ is an  antiholomorphic $(0,1)-$vector field, then $X\lrcorner F^{0}_{A}$ is $\partial_{A}-$closed.
\end{itemize}
Consequently, in either  case, $d_{A}(X\lrcorner F^{0}_{A})$ is $(1,1)$.
\end{lem}
\begin{proof}
It suffices to prove it for  holomorphic $(1,0)$ vector fields, it is similar for  anti-holomorphic $(0,1)$ vector fields. Under a K\"ahler geodesic coordinate, we calculate
\begin{equation}
(X^{i}F^{0}_{i\overline{1}})_{\bar{2}}-(X^{i}F^{0}_{i\overline{2}})_{\bar{1}}=X^{i}F^{0}_{i\overline{1},\bar{2}}-X^{i}F^{0}_{i\overline{2},\bar{1}}=0,
\end{equation}
where the first equal sign holds because $X$ is holomorphic $(1,0)$, the second is by Bianchi identity for $F^{0}_{A}$ and that  the curvature is $(1,1)$. \end{proof}

We henceforth suppress the connection in the derivatives. The other fact is the following. 
\begin{lem}\label{lem contraction is closed}For any constant vector $Y\in \mathbb{R}^{6}$, $d_{0}(rY\lrcorner F^{0}_{A})$ is $(1,1)$. Consequently, $$[d_{0}(rY\lrcorner F^{0}_{A})]\lrcorner G=[d_{0}(rY\lrcorner F^{0}_{A})]\lrcorner H=0.$$ 
\end{lem}

\begin{proof}Let $Z_{0},\ Z_{1},\ Z_{2}$ be the complex coordinates of $\mathbb{C}^{3}$.  It suffices to prove it for the complexified version  for the constant holomorphic vectors 
\begin{equation}\label{equ holo vect fields}\frac{\partial}{\partial Z_{0}},\ \frac{\partial}{\partial Z_{1}},\ \frac{\partial}{\partial Z_{2}}\end{equation} and anti-holomorphic vectors\begin{equation}\label{equ anti holo vect fields}\frac{\partial}{\partial \overline{Z}_{0}},\ \frac{\partial}{\partial \overline{Z}_{1}},\ \frac{\partial}{\partial \overline{Z}_{2}}.\end{equation}
We only do it for $\frac{\partial}{\partial Z_{0}}$ on the dense open set \begin{equation}\label{equ V zariski}V_{\mathbb{C}^{3}\setminus O}=\{Z_{0}\neq 0,\ Z_{1}\neq 0,\ Z_{2}\neq 0\}\subset \mathbb{C}^{3}\setminus O.\end{equation} Then it follows by  continuity.  The proof for the other five vectors are similar. 

We calculate
 \begin{eqnarray}& &d_{0}(r\frac{\partial}{\partial Z_{0}}\lrcorner F^{0}_{A})=d_{0}(\frac{r}{Z_{0}}Z_{0}\frac{\partial}{\partial Z_{0}}\lrcorner F^{0}_{A})\nonumber
 \\&=&\{d_{0}(\frac{r}{Z_{0}})\wedge [\pi_{\star}(Z_{0}\frac{\partial}{\partial Z_{0}})]\lrcorner F^{0}_{A}\}+\frac{r}{Z_{0}}\cdot d_{\mathbb{P}^{2}}([\pi_{\star}(Z_{0}\frac{\partial}{\partial Z_{0}})]\lrcorner F^{0}_{A})
 \label{equ 1 lem contraction is closed}
 \end{eqnarray}
 
 The radius $r$ equals $1$ on $\mathbb{S}^{5}$. According to an identity in \cite[Proof of Lemma 8.7]{W}, $d_{0}(\frac{r}{Z_{0}})=d_{0}(\frac{1}{Z_{0}})$ is $(1,0)$. Since $[\pi_{\star}(Z_{0}\frac{\partial}{\partial Z_{0}})]\lrcorner F^{0}_{A}$ is $(0,1)$,  the first term in \eqref{equ 1 lem contraction is closed} is $(1,1)$. So is the second term by Lemma \ref{equ closed contraction}. 
 \end{proof}
 \subsection{Riemann-Roch formula}
 The map $\boxminus$ being  surjective implies rigidity of the connection. 
\begin{lem}\label{cor rigidity of bundle on P2}Let  $(E,A)\rightarrow \mathbb{P}^{2}$ be a non-projectively flat Hermitian Yang-Mills bundle  with rank $n\geq 2$. Suppose  $c_{2}(EndE)\leq 3.$
Then $n=2$. Moreover,  as a holomorphic bundle, $(E,\overline{\partial}_{A})$ is isomorphic to the twisted Fubini-Study connection on $(T^{1,0}\mathbb{P}^{2})(k)$ for some integer $k$. In particular,   the equality $c_{2}(EndE)=3$ holds.
\end{lem}
In the above case,  we recall that $c_{2}(EndE)=2nc_{2}(E)-(n-1)c^2_{1}(E)$.
\begin{proof}[Proof of Lemma \ref{cor rigidity of bundle on P2}]
The Hermitian Yang-Mills condition implies poly-stability i.e.  $$E=E_{1}\oplus...\oplus E_{m}$$ where the $m$ components are stable bundles of the same slope.  Any (holomorphic) endomorphism of $E$  is determined by the induced homomorphism $$E_{i}\rightarrow E_{j}\ \ \textrm{for any}\ \ i,\ j=1,...,m. $$Then Lemma \ref{lem Homorphism} below yields  a natural complex injection $$H^{0}[\mathbb{P}^{2},EndE]\rightarrow gl(m,\mathbb{C}).$$  This implies   $h^{0}(\mathbb{P}^{2},EndE)\leq m^{2}$. 

Step $1$: We show that $E$ must be stable and rank$E=2$ i.e. $m=1$.  Because   $$h^{0}[\mathbb{P}^{2},(EndE)(-3)]=0,$$ the cohomology formula (for example, see  \cite[Lemma 18.10]{W}) implies 
\begin{equation}\nonumber
0\leq h^{1}[\mathbb{P}^{2},EndE]=2nc_{2}(E)-(n-1)c^2_{1}(E)+(m^{2}-n^{2})\leq m^{2}+3-n^{2}.
\end{equation}
Hence 
\begin{equation}\label{equ rank and number of summands} n^{2}\leq 3+m^{2}.\end{equation} But $$n=n_{1}+...+n_{m}$$ is the sum of the ranks of the sub-bundles. Then either 
\begin{itemize}\item $n_{1}=...=n_{m}=1$,
 \item or $m=1$.
\end{itemize} This is because if $m\geq 2$ and there is at least one summand with rank $\geq 2$, the inequality
$$n^{2}\geq (m+1)^{2}=m^{2}+2m+1>3+m^{2}$$
contradicts  \eqref{equ rank and number of summands}. The first bullet point condition implies $E$ is projectively flat, which contradicts our assumption. The second says 
 $E$ is stable therefore rank$E=2$ by \eqref{equ rank and number of summands}.

Step $2$: It suffices to show $E$ must be a twisted tangent bundle using (the other conditions and) $$0\leq 4c_{2}(E)-c^2_{1}(E)\leq 3.$$ Because the Chern numbers $c_{1}(E)$ and $c_{2}(E)$ are both integers, $4c_{2}(E)-c^2_{1}(E)$ can not be $1$ or $2$ mod $4$.  This is because in mod 4 congruence classes,  $4c_{2}(E)=0$  and the square of an integer (including $c^{2}_{1}(E)$) does  not equal $2$ or $3$. Hence $$4c_{2}(E)-c^2_{1}(E)=\ 3\ \textrm{or}\ 0.$$
  
Case $1$. Suppose $4c_{2}(E)-c^2_{1}(E)=3$. Then $c_{1}(E)$ must be odd. Let $E(k_{E})$ be the normalization of $E$ such that  $
c_{1}[E(k_{E})]=-1,$
we have  $c_{2}[E(k_{E})]=1$. Then $E(k_{E})$ must be topologically isomorphic to $(T^{1,0}\mathbb{P}^{2})(-2)$. Mukai  \cite{Mukai} shows that they must also be  isomorphic as holomorphic bundles.

Case $2$. Suppose $4c_{2}(E)-c^2_{1}(E)=0$.  The equality in Bogomolov inequality is attained. It must be projectively flat, but $E$ is stable with rank $\geq 2$. This is a contradiction.

The above says $E$ must be isomorphic to  $(T^{1,0}\mathbb{P}^{2})(k)$ as  holomorphic vector bundles.  \end{proof}
 Postscript: The reason $c^{2}_{1}(E)$ is a squared integer is that the Picard group of $\mathbb{P}^{2}$ is generated by $O(1)$ and the Chern number $c^{2}_{1}[O(1)]$ is equal to $1$ i.e.
$$\int_{\mathbb{P}^{2}}c_{1}[O(1)]\wedge c_{1}[O(1)]=1\ \ \textrm{as Chern number}.$$ This implies $c^{2}_{1}(E)=(deg E)^{2}$ where $det E=O(deg E)$. 

\subsection{Proof of Proposition \ref{prop atiyah}}
\textbf{Step 1}:   $r(Y\lrcorner F^{0}_{A})$ is  $d_{0}-$co-closed.


We first show it is $d_{\mathbb{C}^{3}}$ co-closed. Similarly to Lemma \ref{lem contraction is closed}, we show the complexified version for the holomorphic and anti-holomorphic vector fields \eqref{equ  holo vect fields}, \eqref{equ anti holo vect fields}. This is straight-forward because the pullback $F^{0}_{A}$ is $(1,1)$ on $\mathbb{C}^{3}\setminus O$, and the projective Hermitian Yang-Mills condition $F^{0}_{A}\lrcorner \frac{d\eta}{2}=0$ on $\mathbb{P}^{2}$ implies $F^{0}_{A}\lrcorner \omega_{\mathbb{C}^{3}}=0$ as pullback. By Bianchi identity, for any $i=0,1,2$, $F^{0}_{A,i\bar{j},j}=0$ on $\mathbb{C}^{3}$. Hence there is a constant $c_{0}$ such that 
\begin{eqnarray}
& &\nonumber d^{\star_{\mathbb{C}^{3}}}_{\mathbb{C}^{3}}( r\frac{\partial }{\partial Z_{i}}\lrcorner F^{0}_{A})= c_{0}(r F^{0}_{A, i\bar{j}})_{j}= c_{0}r_{j} F^{0}_{A, i\bar{j}}+c_{0}(r F^{0}_{A, i\bar{j},j})=  c_{0} F^{0}_{A}(\frac{\partial }{\partial Z_{0}},\frac{\partial}{\partial r})\nonumber
\\&=& 0.
\end{eqnarray}
Because $d^{\star_{\mathbb{C}^{3}}}_{\mathbb{C}^{3}}=\frac{d^{\star_{D^{\star}}}_0}{r^{2}}$ on semi-basic $1-$forms pullback from $\mathbb{S}^{5}$, we find $$d^{\star_{D^{\star}}}_0 (r\frac{\partial}{\partial Z_{i}}\lrcorner F^{0}_{A})=0. $$ 
Similarly proof yields the following for any $i=0,\ 1,\ 2$.  $$d^{\star_{D^{\star}}}_0 (r\frac{\partial}{\partial \overline{Z}_{i}}\lrcorner F^{0}_{A})=0. $$ 
Hence for any constant vector $Y\in \mathbb{R}^{6}$, we have 
 $$d^{\star_{D^{\star}}}_0 (rY\lrcorner F^{0}_{A})=0. $$ 

\textbf{Step 2}: $r(Y\lrcorner F^{0}_{A})$ is an eigen-section of $P$ of eigenvalue $-1$. 

It is semi-basic.  Because $F^{0}_{A_{0}}$ is $(1,1)$ on $\mathbb{C}^{3}\setminus O$ and $\mathbb{R}^{6}$, the $J_{\mathbb{C}^{3}}$ and $J_{0}$ invariance of $F^{0}_{A_{0}}$ tells us that \begin{equation} \label{equ J Atiyah class}[J_{\mathbb{C}^{3}}(Y)]\lrcorner F^{0}_{A_{0}}=J_{\mathbb{C}^{3}}(Y\lrcorner F^{0}_{A_{0}})=J_{0}(Y\lrcorner F^{0}_{A_{0}})\ \textrm{for any}\ Y\in \mathbb{C}^{3}.
 \end{equation}
 We used that on semi-basic vectors and forms, $J_{\mathbb{C}^{3}}$ coincides with $J_{0}$ (see \cite[Appendix]{W}).
Consequently, $$d^{\star_{D^{\star}}}_0 J_{0}(rY\lrcorner F^{0}_{A})=d^{\star_{D^{\star}}}_0 [rJ_{\mathbb{C}^{3}}(Y)\lrcorner F^{0}_{A}]=0$$ for any $Y\in \mathbb{R}^{6}$ as well.  The Lie derivative in the Reeb vector field  is  
\begin{equation}
L_{\upsilon}(rY\lrcorner F^{0}_{A})=rL_{\upsilon}(Y\lrcorner F^{0}_{A})=r(J_{\mathbb{C}^{3}}Y)\lrcorner F^{0}_{A}=rJ_{\mathbb{C}^{3}}(Y\lrcorner F^{0}_{A})=J_{0}(rY\lrcorner F^{0}_{A}).
\end{equation}
Apply $J_{0}$ to both hand sides and using that $L_{\upsilon}J_{0}=J_{0}L_{\upsilon}$, we find
\begin{equation}
L_{\upsilon}[J_{0}(rY\lrcorner F^{0}_{A})]=-rY\lrcorner F^{0}_{A}.
\end{equation}
Via the formula for $P$ in \cite[Lemma 5.3]{W}, the above means $rY\lrcorner F^{0}_{A}$ is an eigen-section of $P$ of eigenvalue $-1$.  
 It defines the map $$\boxminus: \mathbb{R}^{6}\rightarrow Eigen_{-1}P\ \ \textrm{via}\ \ \ \boxminus(Y)\triangleq rY\lrcorner F^{0}_{A}.$$ 

\textbf{Step 3}. $\boxminus$ is injective.  

Let  $p$ be a point on  $\mathbb{S}^{5}$ at which $v=(rY)^{\parallel_{D}}$  is nonzero (see Fact \ref{fact rY has non-vanishing D component on a dense open set} below). Then we normalize  it via $e_{1}=\frac{v}{|v|}$, and complete it into an orthonormal frame $(e_{1},\ e_{2},\ e_{3},\ e_{4})$ for the contact distribution $D$ at $p$. That $F^{0}_{A}$ is anti self-dual means 
\begin{equation}F^{0}_{A}=F^{0}_{A,I}(e^{12}-e^{34})+F^{0}_{A,II}(e^{13}-e^{42})+F^{0}_{A,III}(e^{14}-e^{23}).
\end{equation}
The condition $e_{1}\lrcorner F^{0}_{A}=0$ at $p$ implies that $0=F^{0}_{A,I}e^{2}+F^{0}_{A,II}e^{3}+F^{0}_{A,III}e^{4}.$ This in turn implies $F^{0}_{A,I}=F^{0}_{A,II}=F^{0}_{A,III}=0$ i.e. $F^{0}_{A}=0$ at $p$. Because $v$ is non-zero on a dense open set on $\mathbb{S}^{5}$, $F^{0}_{A}=0$ on the same set. By continuity of $F^{0}_{A}$, it vanishes everywhere on $\mathbb{S}^{5}$. 

When $E$ is a twisted tangent bundle of $\mathbb{P}^{2}$, by Lemma \ref{cor rigidity of bundle on P2}, the injection $\boxminus$ is  an isomorphism since the dimension of the domain equals the dimension of the range. 
The proof is complete.

\section{Formula of the auxiliary operator\label{Sect Aux}}
The Atiyah classes originally defined in  $Eigen_{-1}P$ 
can also be defined in  $Eigen_{-2}P$ via the isometry $\underline{T}$ i.e. 
$$\{\textrm{Atiyah classes}\}|_{Eigen_{-2}P} \triangleq \underline{T}[\{\textrm{Atiyah classes}\}|_{Eigen_{-1}P}].$$
The desired formula involves both.

 \begin{prop}\label{lem equ for X vs equ for u and v} Let $X\in C^{1}[(B(R)\setminus O)\times \mathbb{S}^{1}]$ be a  vector field of the form \eqref{equ vector field}, $0<R\leq \infty$.   The following holds therein. 
 \begin{equation}\label{eq Aux}Aux(X)=-\Sigma_{i=1}^{6}\{\frac{\partial X_{i}}{\partial s} \cdot  [(J_{\mathbb{C}^{3}}e_{i})\lrcorner F^{0}_{A}]+\frac{\partial X_{i}}{\partial r} \cdot  J_{H}(e_{i}\lrcorner F^{0}_{A})\}.\end{equation}
 Consequently, $RangeAux$ lies in the span (by continuous functions only of $r,s$  on the same tubular ball) of Atiyah classes in $Eigen_{-1}P$ and $Eigen_{-2}P$.
  \end{prop}

 \begin{proof}[Proof of Lemma \ref{lem equ for X vs equ for u and v}:] It suffices to apply Lemma \ref{lem formula for auxiliary operator}  and calculate the Lie derivatives therein. The condition that the co-efficients of $X$ only depend on $r,\ s$ is also used in Step 3 of the proof of the preliminary formula \eqref{equ auxiliary with F} below.

 Because the $\frac{\partial}{\partial s}-$component does not contribute to the operator at all (see Fact \ref{fact srv do not contribute to Aux} below), we can assume $X$ is perpendicular to $\frac{\partial}{\partial s}$. The Lie derivatives in the Reeb-vector field $\upsilon$,  radial vector field $\frac{\partial}{\partial r}$, and $\frac{\partial}{\partial s}$ of the symmetric bi-linear forms $ds^{2}$, $dr^{2}$, and $g_{\mathbb{S}^{5}}$ all vanish. We note that $\frac{\partial}{\partial r}$ is not Killing for the Euclidean  metric $ds^{2}+dr^{2}+r^{2}g_{\mathbb{S}^{5}}$. We compute via the Lie derivative formulas in Lemma \ref{formula Lxi is -J} and  elementary Riemannian geometry that
 \begin{equation}L_{\frac{\partial}{\partial r}}X =\Sigma_{i=1}^{6}[\frac{\partial X_{i}}{\partial r}e_{i}-X_{i}\frac{e_{i}^{\perp_{\frac{\partial}{\partial r}}}}{r}]
 ,\ L_{\upsilon}X =-\Sigma_{i=1}^{6}X_{i} J_{\mathbb{C}^{3}}(e_{i})
,\ L_{\frac{\partial}{\partial s}}X = \Sigma_{i=1}^{6}\frac{\partial X_{i}}{\partial s}e_{i}. \end{equation}

 Then the $3$ Lie-derivative contractions  in \eqref{equ auxiliary with F} can be calculated as follows. 
 \begin{eqnarray*}& &(L_{\frac{\partial}{\partial s}}X)\lrcorner \frac{d\eta}{2}=\Sigma_{i=1}^{6}\frac{\partial X_{i}}{\partial s}(e_{i}\lrcorner \frac{d\eta}{2}),
\label{equ  lem equ for X vs equ for u and v 0}
\\& &(L_{\frac{\partial}{\partial r}}X)\lrcorner H=\Sigma_{i=1}^{6}(\frac{\partial X_{i}}{\partial r}-\frac{X_{i}}{r})(e_{i}\lrcorner H),\
\\& & \frac{1}{r}(L_{\upsilon}X)\lrcorner G=-\Sigma_{i=1}^{6}\frac{X_{i}}{r}[J_{\mathbb{C}^{3}}(e_{i})\lrcorner G]
 =\Sigma_{i=1}^{6}\frac{X_{i}}{r}(e_{i}\lrcorner H).  \nonumber
 \end{eqnarray*}
Summing the above $3$ equalities and combining coefficients of similar terms, the two terms containing $e_{i}\lrcorner H$ cancel out and we find 
   \begin{equation}(\frac{\partial X}{\partial s})\lrcorner \frac{d\eta}{2}+(L_{\frac{\partial}{\partial r}} X)\lrcorner H+\frac{(L_{\upsilon}X)\lrcorner G}{r}=\Sigma_{i=1}^{6}\{\frac{\partial X_{i}}{\partial s} \cdot  (e_{i}\lrcorner \frac{d\eta}{2})+\frac{\partial X_{i}}{\partial r} \cdot  (e_{i}\lrcorner H)\}.\nonumber
   \end{equation}
   Here we applied again the remark below \eqref{equ J Atiyah class} about the relation between  $J_{\mathbb{C}^{3}}$ and $J_{0}$.
Using  \eqref{equ auxiliary with F} and contracting the above with $-F^{0}_{A}$, the proof of the desired formula \eqref{eq Aux} is complete.
 \end{proof}

\section{The Dirac system and  proof of Theorem \textbf{A} and Corollary \textbf{B} \label{proof of Theorem A}}
In this section we assemble the established tools to prove the main results. Via separation of variables, the singular sequence is constructed via a linear system of two partial differential equations in $r$ and $s$. 
\subsection{The Dirac system}
 Let $(\xi,X)$ be the independent variable of the model extended linearized operator $L_{0}$, where $X$ is the vector field and $\xi$ is the section of the domain bundle $adE\oplus \Omega^{1}_{adE}$. Because \textit{$RangeAux$ is  spanned  by functions in $r$ and $s$ of Atiyah classes} in both $Eigen_{-1}P$ and $Eigen_{-2}P$, in the perpendicular direction, the extended linearized operator coincides with the usual linearized operator in the following sense. In the Hilbert space $L^{2}[\mathbb{S}^{5}, Dom]$, let $\parallel_{Atiyah}$ denote the projection to the $12$ dimensional subspace 
 \begin{equation}\nonumber
 \{\textrm{Atiyah classes}\}|_{Eigen_{-1}P}\oplus \{\textrm{Atiyah classes}\}|_{Eigen_{-2}P},
 \end{equation}
  and $\perp_{Atiyah}$ the projection to the orthogonal complement.  We have 
$$[Aux(X)]^{\parallel_{Atiyah}}=Aux(X) $$
  for any differentiable vector field $X$ in the punched tubular ball.
  
   For any $\xi\in C^{1}[(B(R_{0})\setminus O)\times \mathbb{S}^{1}]$, 
 \begin{equation}\label{equ perp Atiyah}
 [L_{0}(\xi,X)]^{\perp_{Atiyah}}=(\underline{L}_{0}\xi)^{\perp_{Atiyah}}=\underline{L}_{0}(\xi^{\perp_{Atiyah}}).
 \end{equation}
But $Aux$ does  appear  in the Atiyah class component:
   \begin{equation}
 [L_{0}(\xi,X)]^{\parallel_{Atiyah}}=(\underline{L}_{0}\xi)^{\parallel_{Atiyah}}+Aux(X)=\underline{L}_{0}(\xi)^{\parallel_{Atiyah}}+Aux(X). 
 \end{equation}

Gram-Schmit process for each eigen-space of $P$ yields a complete orthonormal $P-$eigen-basis $(\phi_{\beta},\ \beta\in Spec^{mul}P)$  for $L^{2}[\mathbb{S}^{5},Dom]$ such that

\begin{itemize}\item  the eigen-section $\zeta$ perpendicular to the Atiyah classes  in condition III appears  as an element in the basis  if essential obstruction is non-trivial, 
\item  $6$ elements of the eigen-basis form an $I-$invariant  orthonormal basis for\\ $\{\textrm{Atiyah  classes}\} |_{Eigen_{-1}P}$, and  applying  $\underline{T}$  yields that of $\{\textrm{Atiyah classes}\}|_{Eigen_{-2}P}$.
\end{itemize}

Via separation of variables, we need to solve equations for the Fourier-coefficient of  an arbitrary section $\phi_{\beta}$ in the eigen-basis.
 However, because of the endomorphism $\underline{T}$ in the  Dirac operator \eqref{equ underlineL0} (see \cite[Lemma 5.3]{W}),  we need to consider $\phi_{\beta}$ and $\underline{T}\phi_{\beta}$ simultaneously. Particularly, in line with \eqref{equ perp Atiyah} and that $\zeta$ is perpendicular to the Atiyah classes,   the operator $-I\cdot \underline{L}_{0}$ also preserves the span by functions in $r,\ s$ of $\zeta$ and $\underline{T}\zeta$: 
  \begin{equation}\label{equ L0 span}
 (-I\cdot \underline{L}_{0}\xi)^{\parallel_{span\{\zeta,\ \underline{T}\zeta\}}}=(-I\cdot \underline{L}_{0})(\xi)^{\parallel_{span\{\zeta,\ \underline{T}\zeta\}}}\ \textrm{for any}\ \xi\in C^{1}[(B(R_{0})\setminus O)\times \mathbb{S}^{1}].
 \end{equation}
 
 The   equation in $span\{\zeta,\ \underline{T}\zeta\}$ of two unknowns $x$ and $y$ reads
$$-I\cdot L(x\zeta+y\underline{T}\zeta)=f\zeta+g\underline{T}\zeta.$$
According to formula \eqref{equ underlineL0} for the usual linearized operator, 
this is equivalent to the  \textit{Dirac system} of two variables:
 \begin{equation}\left\{\begin{array}{c}\frac{\partial x}{\partial s}+\frac{\partial y}{\partial r}+\frac{2y}{r}=f;\\
 \frac{\partial y}{\partial s}-\frac{\partial x}{\partial r}-\frac{ x}{r}=g.
 \end{array}\right.
 \end{equation}

Plugging
\begin{equation}\label{equ FDy}\frac{\partial y}{\partial s}=\frac{\partial x}{\partial r}+\frac{x}{r}+g. \end{equation}into $\frac{\partial}{\partial s}$ of the first equation, we find a second order equation only in $x$.
\begin{equation}\nonumber
\frac{\partial^{2} x}{\partial r^{2}}+\frac{\partial^{2} x}{\partial s^{2}}+\frac{3}{r}\frac{\partial x}{\partial r}+\frac{x}{r^{2}}=\frac{\partial f}{\partial s}-\frac{\partial g}{\partial r}-\frac{2g}{r}\triangleq h.
\end{equation}
The equation of the Fourier co-efficient of $\cos ks$ and $\sin ks$ reads
 \begin{equation}\label{equ FDx} x^{\prime\prime}_{k}+\frac{3 x^{\prime}_{k}}{r}+\frac{x_{k}}{r^{2}}-k^{2}x_{k}=h_{k}.
 \end{equation}
 This ordinary differential equation  can be solved elementarily.

\subsection{The singular sequence \label{sect singular seq and proof of main thm}}
Now we construct a sequence that violates the closeness of the range. We only consider positive independent variable for the special functions. Let $k$ be a positive integer and 
\begin{equation}\label{equ bad h}h_{k}(y)\triangleq y^{\delta}\chi(ky)K_{0}(ky),\end{equation} where $\chi(r)$ is a cut-off function that is $\equiv 0$ when $r\leq 1$ or $r\geq 4$, but $\equiv 1$ when $r\in [2,3]$.

In the following, $h_{k}$ and $x_{k}$ are specific as \eqref{equ bad h} and \eqref{equ lem bad solution}, but in the previous section they are general. 
\begin{lem}\label{lem bad solution go to infty when k goes to infinity} For any non-negative number $\delta$  there is a positive constant $C_{\delta}$ with the following property.  Let $h_{k}$ be as  
\eqref{equ bad h}.  The only solution to \eqref{equ FDx} that $=O(1)$ as $r\rightarrow 0$  is 
\begin{equation}\label{equ lem bad solution}x_{k}=-\frac{K_{0}(kr)}{r}\int^{r}_{0}I_{0}(ky)y^{2}h_{k}(y)dy+\frac{I_{0}(kr)}{r}\int^{r}_{0}K_{0}(ky)y^{2}h_{k}(y)dy.
\end{equation}
 The following holds for any positive integer $k$ and  real number $r$ such that $kr\geq 10$.
\begin{equation}\label{equ xk lower bound}|x_{k}|\geq  \frac{C_{\delta}\cdot e^{kr}}{k^{\frac{7}{2}+\delta}r^{\frac{3}{2}}}.\end{equation} 
Consequently, $\lim_{k\rightarrow \infty}|x_{k}|=+\infty$ uniformly on any compact subset of $(0,\infty)$.

\end{lem}
\begin{rmk} The solution $x_{k}$ is  supported away from $0$ since $h_{k}$ is. The constant $C_{\delta}$ is given by integral and point-wise bounds on the special functions. 
\end{rmk}
 
  
\begin{proof}The trick is to consider  $kr$ instead of $r$ alone. 
The general solution to the ODE is
\begin{equation}
-\frac{K_{0}(kr)}{r}\int^{r}_{0}I_{0}(ky)y^{2}h_{k}(y)dy+\frac{I_{0}(kr)}{r}\int^{r}_{0}K_{0}(ky)y^{2}h_{k}(y)dy+\frac{aK_{0}(kr)}{r}+\frac{bI_{0}(ky)}{r}.
\end{equation}
The main part $x_{k}$ is compactly supported away from $0$, but the homogeneous solutions $\frac{K_{0}(kr)}{r}$ and $ \frac{I_{0}(kr)}{r}$ have leading terms $\frac{C \log r}{r}$ and $\frac{C}{r}$ for nonzero constant $C's$, respectively. Since we require $x$ to be $O(1)$,   these two homogeneous solutions can  not appear i.e.   $a,\ b$ must be $0$.


 In order to bound the first term in \eqref{equ lem bad solution}, we estimate the integral for any $r$: 
\begin{equation}
|\int^{r}_{0}I_{0}(ky)y^{2}h_{k}(y)dy|\leq \frac{1}{k^{3+\delta}}\int^{\infty}_{0}\chi(w)I_{0}(w)|K_{0}(w)|w^{2+\delta}dw\leq \frac{\underline{C}_{2,\delta}}{k^{3+\delta}}, 
\end{equation}
where  $\underline{C}_{2,\delta}$ is  the value of the integral $$\int_{1}^{4}I_{0}(w)|K_{0}(w)|w^{2+\delta}dw.$$ Please notice that $\chi$ is supported in the interval. Then if  $kr\geq 1$, using  the bound on $\frac{K_{0}(x)}{x}$ when $x\geq 1$,   we find 
\begin{equation}|-\frac{K_{0}(kr)}{r}\int^{r}_{0}I_{0}(ky)h_{k}(y)y^{2}dy|\leq \frac{\underline{C}_{2,\delta}}{k^{2+\delta}}\cdot \frac{|K_{0}(kr)|}{kr} \leq \frac{\underline{C}_{3,\delta}}{k^{2+\delta}}.\end{equation}

To bound the second term in \eqref{equ lem bad solution} from below, we compute 
 \begin{eqnarray*}
& &\frac{I_{0}(kr)}{r}\int^{r}_{0}K_{0}(ky)y^{2}h_{k}(y)dy=\frac{I_{0}(kr)}{k^{3+\delta}r}\int^{kr}_{0}\chi(w)K^{2}_{0}(w)w^{2+\delta} dw\geq \frac{\underline{C}_{4,\delta}\cdot I_{0}(kr)}{k^{3+\delta}r}\\&\geq &\frac{\underline{C}_{4,\delta}\cdot \underline{C}_{5} e^{kr}}{k^{3+\delta}r\cdot \sqrt{kr}}. \nonumber
\end{eqnarray*}
where the constant $\underline{C}_{4,\delta}$ equals $\int ^{4}_{1}K^{2}_{0}(w)w^{2+\delta}dw$, and $\underline{C}_{5}$ equals the positive lower bound on $e^{-w}\sqrt{w}I_{0}(w)$ for $w\geq 1$. Let $C_{\delta}$ be large enough regarding these two constants and $\underline{C}_{3,\delta}$, the proof of \eqref{equ xk lower bound} is complete.  \end{proof}

\subsection{Proof of Theorem \textbf{A} and Corollary \textbf{B}}
In functional analysis, closed range is equivalent to  existence of ``a priori estimate" in the following sense. 

\begin{fact}\label{fact closed range}
Suppose $L:\ X\rightarrow Y$ is a bounded linear map between Banach spaces. Then $RangeL$ is closed if and only if there is a non-negative constant $N$ such that for any $y\in RangeL$, there exists a solution $x$ to the equation $Lx=y$ with the bound 
\begin{equation}\label{equ Functorial uniform est}||x||_{X}\leq N||y||_{Y}.\end{equation} 
\end{fact}
\begin{proof}[\textbf{Proof of Theorem} \textbf{A}:]  
\textit{The idea is to construct singular sequence whenever the essential obstruction does not vanish}.  We only show it for  the extended linearized operator using conditions II--V and III$^\star$.  Similar argument applies to the usual $\underline{L}$ under conditions I--III and III$^\star$. 

Definition \ref{Def gdc} of the configuration says that we are in the model setting in the tubular ball. 
Given a large enough positive integer $k$, we specify the single variable function $h_{k}$ in $y$ (the radius) as in \eqref{equ bad h} and let  $f_{k}=\frac{h_{k}}{k}$. Again, let $\zeta$ be the eigen-section in condition III.  
 
\textit{Because the auxiliary operator does not cover $\zeta$} which is perpendicular to all the Atiyah classes in $Eigen_{-1}P$, and that 
 $-I\cdot L_{0}$ commutes with  the projection to  $span\{\zeta,\ \underline{T}\zeta\}$ ($-I\cdot L_{0}=-I\cdot \underline{L}_{0}$ thereon, see \eqref{equ perp Atiyah} and \eqref{equ L0 span}), the $\zeta \cos ks-$component of any solution $\xi_{k}= O(\frac{1}{r})$ to 
 \begin{equation}\label{equ 0 proof Thm A}
 L_{0}\xi_{k}=(f_{k}\sin ks )I\zeta
 \end{equation}
 must be $O(1)$ thus equals the $x_{k}$ in \eqref{equ lem bad solution}. To see this, in view of the argument  from \eqref{equ L0 span} to \eqref{equ FDx} on Dirac system, we simply project both sides of  \eqref{equ 0 proof Thm A} onto $span\{\zeta,\ \underline{T}\zeta\}$ according to  \eqref{equ L0 span}, then take the Fourier co-efficient of $\cos ks$ and apply Lemma \ref{lem bad solution go to infty when k goes to infinity}.  
Therefore the $L^{2}-$norm of $\xi_{k}$ on the stripe defined by $\underline{R_{0}}/10<r< \underline{R_{0}}/5$ tends to $\infty$ as $k\rightarrow \infty$. As the $C^{0}_{1}-$norm on the punched ball of radius $\underline{R}_{0}$ is stronger than this $L^{2}-$norm, condition II implies
$$|\xi_{k}|_{Y}\rightarrow \infty \ \textrm{as}\ k\rightarrow \infty.$$ 
Condition III$^\star$ says that $(f_{k}\sin ks) I\zeta$ is in $RangeL$ and III says their $\mathcal{B}-$norm are uniformly bounded. According to the characterization of closed range  in Fact \ref{fact closed range},  $RangeL$ is not closed. \end{proof}

Under a tame configuration over a compact $7-$fold, let  \begin{equation}C_{r,s}^{2}[M, TM]\subset C^{2}[M, TM]  \label{equ vect fields for the example}
\end{equation}
be the subspace of $C^{2}$ vector fields that restrict to the  form \eqref{equ vector field} in the punched tubular ball  i.e. only depending on $r$ and $s$  in $B(R_{0})\times \mathbb{S}^{1}$ near each circle. Between weighted Schauder spaces as in Section \ref{sect weighted schauder},   consider  the usual linearized operator \begin{equation}\label{equ usual L compact}
  \underline{L}:\ C^{1,\frac{1}{2}}_{1-\delta}[M^{7}\setminus\gamma, \Omega^{0}_{adE}\oplus \Omega^{1}_{adE}]\rightarrow  C^{1,\frac{1}{2}}_{2-\delta}[M^{7}\setminus\gamma, \Omega^{0}_{adE}\oplus \Omega^{1}_{adE}]
\end{equation} and the extended linearized operator 
 \begin{equation}\label{equ extended  usual L compact}
L:\ C^{1,\frac{1}{2}}_{1}[M\setminus\gamma, \Omega^{0}_{adE}\oplus \Omega^{1}_{adE}]\times C_{r,s}^{2}[M, TM] \rightarrow  C^{1,\frac{1}{2}}_{2}[M^{7}\setminus\gamma, \Omega^{0}_{adE}\oplus \Omega^{1}_{adE}].
\end{equation}  
Both are bounded. 
The reason we let $\delta=0$ for the extended linearization is that we do not know whether $Aux$  has a better bound than $C^{1,\frac{1}{2}}_{2}$, due to the quadratic growth of the norm of the curvature near the circles.
\begin{proof}[\textbf{Proof of Corollary B}:] It is straight-forward to verify conditions I--V. Was $L$ surjective, condition III$^{\star}$ holds as well i.e. the configuration is ideal. Then  Theorem \textbf{A} says $RangeL$ is not closed, which is a contradiction. Similar argument applies to the usual $\underline{L}$.

For the reader's convenience, we still provide the detail in checking the conditions.   


\begin{itemize}
\item Condition I (saying $\underline{L}$ is bounded)   holds by formula \eqref{equ introduction formula for deformation operator}, our choice \eqref{equ extended  usual L compact}, and definition of the weighted Schauder spaces. The weight for the first derivatives has $1$ more power than that for the section itself. 


\item Condition II (coerciveness) holds because restricted to the tubular ball, the norm $C^{0}_{1}$ is weaker than $C^{1,\frac{1}{2}}_{1}$ ($C^{1,\frac{1}{2}}_{1-\delta}$).
\item Condition III (bound on the particular sequence) follows simply from  the decay of the modified Bessel function $K_{0}$ and that  $\chi$ is non-negative, supported in $(1, 4)$, and bounded by $1$. Namely, the following holds for large positive integer $k$.  
\begin{equation}\nonumber
\sup_{r\in (0,\infty)}r^{2-\delta}|\frac{r^{\delta}\chi(kr)K_{0}(kr)I\zeta\cdot \sin ks}{k}| \nonumber
= \sup_{r\in (0,\infty)}|\frac{(kr)\chi(kr)K_{0}(kr)}{k^{2}}|\cdot |rI\zeta| 
\leq  C.
\end{equation}
The  $r^{3-\delta}-$weighted  bounds on the $\frac{\partial}{\partial r}$,  $\frac{\partial }{\partial s}$, and $\frac{\nabla_{\mathbb{S}^{5}}}{r}$ of $\frac{r^{\delta}\chi(kr)K_{0}(kr)I\zeta\cdot \sin ks}{k}$ follow similarly. Then 
\begin{equation}
|\frac{r^{\delta}\chi(kr)K_{0}(kr)I\zeta\cdot \sin ks}{k}|_{C^{1}_{2-\delta}}\leq C.
\end{equation}
This implies the $C^{0,\frac{1}{2}}_{2-\delta}-$bound  of the same thing by interpolation of weighted Schauder norms. 
\item Condition III$^{\star}$ is simply the contradiction hypothesis that the linearization is surjective. Condition IV holds automatically because of  our vector fields \eqref{equ vect fields for the example}.
\item   Condition V  (saying $L$ is bounded)  holds by formula \eqref{equ l tilde}, our choice \eqref{equ extended  usual L compact},  and the simple weighted H\"older bound on the auxiliary operator:
   \begin{equation}|\star(F^{0}_{A}\wedge d[X\lrcorner \psi])+\star(F^{0}_{A}\wedge [X\lrcorner d\psi])|_{C^{0,\frac{1}{2}}_{2}}\leq |X|_{C^{2}}.\nonumber
\end{equation}
Because it involves first partial  derivatives of $X$, we need $X$ to be $C^{2}$. 
\end{itemize}
\end{proof}
\section{Appendix}
 \addtocontents{toc}{\protect\setcounter{tocdepth}{0}}
\subsection{Non-vanishing of  a certain projection of coordinate vector fields on $\mathbb{R}^{6}$}
It is routine to check the following ``non-vanishing" that applies to the proof of Proposition \ref{prop atiyah}.  Let $(Z_{0},Z_{1},Z_{2})$ be the coordinates for $\mathbb{C}^{3}$ and $\upsilon$ be the standard Reeb vector field on $\mathbb{S}^{5}$.
 \begin{fact}\label{fact rY has non-vanishing D component on a dense open set}Let $Y\in \mathbb{R}^{6}\setminus O$ be a (constant) nonzero vector. There exists a dense open set on $\mathbb{S}^{5}$ on which $(rY)^{\parallel_{D}}$ is non-zero everywhere. 
\end{fact}
\begin{proof}[Proof of Fact \ref{fact rY has non-vanishing D component on a dense open set}:] We write $Y=Y^{1,0}+Y^{0,1}$, where $$Y^{1,0}=a_{Y}\frac{\partial}{\partial Z_{0}}+b_{Y}\frac{\partial}{\partial Z_{1}}+c_{Y}\frac{\partial}{\partial Z_{2}},\ \textrm{and}\ Y^{0,1}=\overline{Y^{1,0}}$$
for some  complex constants $a_{Y},\ b_{Y},\ c_{Y}$. Under the Sasakian coordinate in $U_{0,\mathbb{S}^{5}}\subset \mathbb{S}^{5}$ defined by $Z_{0}\neq 0$ \cite[(15)]{W}, 
using formula \eqref{equ formula for the invariant vector field ZddZ} and  $$Z_{0}\frac{\partial}{\partial Z_{1}}=\frac{Z_{0}\bar{Z}_{1}}{2r}\frac{\partial}{\partial r}+\frac{\partial}{\partial u_{1}};\ Z_{0}\frac{\partial}{\partial Z_{2}}=\frac{Z_{0}\bar{Z}_{2}}{2r}\frac{\partial}{\partial r}+\frac{\partial}{\partial u_{2}}$$ for $(1,0)$ coordinate vectors in $\mathbb{C}^{3}$, 
we calculate the projection onto the contact distribution over $\mathbb{S}^{5}$:
\begin{eqnarray*} (rY^{1,0})^{\parallel_{D}}&=&\frac{r}{Z_{0}}[(b_{Y}-a_{Y}u_{1})\frac{\partial}{\partial u_{1}}+(c_{Y}-a_{Y}u_{2})\frac{\partial}{\partial u_{2}}]^{\parallel_{D}}
\\&=& \frac{r}{Z_{0}}[(b_{Y}-a_{Y}u_{1})(\frac{\partial}{\partial u_{1}}-\eta(\frac{\partial}{\partial u_{1}})\upsilon)+(c_{Y}-a_{Y}u_{2})(\frac{\partial}{\partial u_{2}}-\eta(\frac{\partial}{\partial u_{2}})\upsilon)].\end{eqnarray*}
When  $(b_{Y}-a_{Y}u_{1})\neq 0$ or $(c_{Y}-a_{Y}u_{2})\neq 0$, we have $(rY^{1,0})^{\parallel_{D}}\neq 0$. These two non-vanishing conditions together with $Z_{0}\neq 0$ define a dense open set on $\mathbb{S}^{5}$.
\end{proof}




\subsection{The Lie derivatives of the vector fields on $\mathbb{C}^{3}\setminus O$}

Proposition \ref{lem equ for X vs equ for u and v} applies the following formulas of the Lie derivatives. 
\begin{lem}\label{formula Lxi is -J}1. Let $\upsilon$ be the standard Reeb vector field on $\mathbb{S}^{5}$. Then \begin{equation}\label{equ 0 formula Lxi is -J}
L_{\upsilon}(Z_{i}\frac{\partial}{\partial Z_{i}})=0,\ \ L_{\upsilon}(\bar{Z}_{i}\frac{\partial}{\partial \bar{Z}_{i}})=0.
\end{equation} 
Consequently, on $\mathbb{R}^{6}$ and its complexfication, $L_{\upsilon}=-J_{\mathbb{C}^{3}}$ is equal to the negative of the standard complex structure i.e.  
\begin{equation}\label{equ 1 formula Lxi is -J} L_{\upsilon}\frac{\partial}{\partial Z_{i}}=-\sqrt{-1}\frac{\partial}{\partial Z_{i}},\ L_{\upsilon}\frac{\partial}{\partial \bar{Z}_{i}}=\sqrt{-1}\frac{\partial}{\partial \bar{Z}_{i}}.\end{equation}
Particularly, for any vector $Y\in \mathbb{R}^{6}$, $L_{\upsilon}=-J_{\mathbb{C}^{3}}Y$.

2. $L_{\frac{\partial}{\partial r}}\frac{\partial}{\partial Z_{i}}=-\frac{(\frac{\partial}{\partial Z_{i}})^{\perp_{\frac{\partial}{\partial r}}}}{r}$. The complex conjugate  $L_{\frac{\partial}{\partial r}}\frac{\partial}{\partial \overline{Z}_{i}}=-\frac{(\frac{\partial}{\partial \overline{Z}_{i}})^{\perp_{\frac{\partial}{\partial r}}}}{r}$ also holds. This means $L_{\frac{\partial}{\partial r}}$ is $-\frac{1}{r}$ times the projection to the orthogonal complement of $\frac{\partial}{\partial r}$. 
Particularly, for any vector $Y\in \mathbb{R}^{6}$, $L_{\frac{\partial}{\partial r}}Y=-\frac{1}{r}\cdot Y^{\perp_{\frac{\partial}{\partial r}}}$.

\end{lem}

\begin{proof} We  prove them in the Zariski open set $V_{\mathbb{C}^{3}\setminus O}$ \eqref{equ V zariski}. Then the global equations follow  by  continuity. 

On item 1, recall again \cite[(15)]{W} about the Sasakian coordinate system and the local K\"ahler potentials  $\phi_{i}$ of the  Fubini-Study metric $\frac{d\eta}{2}$ such that  \begin{equation}\label{equ 3 formula Lxi is -J}Z_{i}=\frac{r}{\sqrt{\phi_{i}}}e^{i\theta_{i}}.\end{equation}
In the $i-th$ Sasakian coordinate chart, the Reeb vector field  $\upsilon$ equals $\frac{\partial}{\partial \theta_{i}}$ (\cite[Fact 3.4]{W}). On the Euler sequence, $Z_{i}\frac{\partial }{\partial Z_{i}}$ are the scaling invariant holomorphic  vector fields on $\mathbb{C}^{3}\setminus O$ whose projections to $\mathbb{P}^{2}$ span the holomorphic tangent bundle point-wisely.  We directly verify the following identities.
 \begin{eqnarray}\label{equ formula for the invariant vector field ZddZ}& &Z_{0}\frac{\partial}{\partial Z_{0}}=\frac{r}{2\phi_{0}}\frac{\partial}{\partial r}-\frac{\sqrt{-1}}{2}\frac{\partial}{\partial \theta_{0}}-u_{1}\frac{\partial}{\partial u_{1}}-u_{2}\frac{\partial}{\partial u_{2}}\ \textrm{in}\ \ U_{0,\mathbb{C}^{3}};\\ & & \label{equ formula for the invariant vector field ZddZ 1}
Z_{1}\frac{\partial}{\partial Z_{1}}=\frac{r}{2\phi_{1}}\frac{\partial}{\partial r}-\frac{\sqrt{-1}}{2}\frac{\partial}{\partial \theta_{1}}-v_{0}\frac{\partial}{\partial v_{0}}-v_{2}\frac{\partial}{\partial v_{2}}\ \textrm{in}\ \ U_{1,\mathbb{C}^{3}};\\& & \label{equ formula for the invariant vector field ZddZ 2}
Z_{2}\frac{\partial}{\partial Z_{2}}=\frac{r}{2\phi_{2}}\frac{\partial}{\partial r}-\frac{\sqrt{-1}}{2}\frac{\partial}{\partial \theta_{2}}-w_{0}\frac{\partial}{\partial w_{0}}-w_{1}\frac{\partial}{\partial w_{1}} \ \textrm{in}\ \ U_{2,\mathbb{C}^{3}}.
\end{eqnarray}

Using the above identities,  the desired \eqref{equ 0 formula Lxi is -J} follows because each term in \eqref{equ formula for the invariant vector field ZddZ}--\eqref{equ formula for the invariant vector field ZddZ 2} has vanishing Lie bracket with the Reeb vector field. By \eqref{equ 3 formula Lxi is -J} and the characterization of $\upsilon$ above, we simply obtain  
$$L_{\upsilon}Z_{i}=\sqrt{-1}Z_{i}$$
from which \eqref{equ 1 formula Lxi is -J} follows.  

 We now prove item 2. 
\begin{eqnarray}& &L_{\frac{\partial}{\partial r}}\frac{\partial}{\partial Z_{i}}=[\frac{\partial}{\partial r},\frac{1}{Z_{i}}(Z_{i}\frac{\partial}{\partial Z_{i}})]=[\frac{\partial}{\partial r}(\frac{1}{Z_{i}})]\cdot (Z_{i}\frac{\partial}{\partial Z_{i}})+\frac{1}{Z_{i}}[\frac{\partial}{\partial r},\ Z_{i}\frac{\partial}{\partial Z_{i}}] \nonumber
\\&=& -\frac{1}{rZ_{i}}(Z_{i}\frac{\partial}{\partial Z_{i}})+\frac{1}{2\phi_{i}Z_{i}}\frac{\partial}{\partial r}= -\frac{1}{r}\frac{\partial}{\partial Z_{i}}+\frac{1}{2\phi_{i}Z_{i}}\frac{\partial}{\partial r}\nonumber
\\&=&  -\frac{(\frac{\partial}{\partial Z_{i}})^{\perp_{\frac{\partial}{\partial r}}}}{r},\nonumber
\end{eqnarray}
where we used that the orthogonal projection of $\frac{\partial }{\partial Z_{i}}$ to $\frac{\partial}{\partial r}$ is $\frac{r}{2\phi_{i}Z_{i}}\frac{\partial}{\partial r}$. 
\end{proof}
\subsection{On  the Auxiliary operator \label{App formula for Aux}}
We provide the routine tensor calculation for  Proposition \ref{lem equ for X vs equ for u and v}. 

\begin{lem}\label{lem formula for auxiliary operator} Under the conditions in Proposition \ref{lem equ for X vs equ for u and v} and  the splitting of tangent bundle $$T[(\mathbb{C}^{3}\setminus O)\times \mathbb{S}^{1}]=Span(\frac{\partial}{\partial s},\frac{\partial}{\partial r},\upsilon)\oplus D$$ where $D$ is the contact distribution on $\mathbb{S}^{5}$ and $\upsilon$ is the Reeb vector field, we write the vector field (whose co-efficients under the standard Euclidean basis only depend on $r,\ s$, see our assumption \eqref{equ vector field}  as $$X=X_{s}\frac{\partial}{\partial s}+X_{r}\frac{\partial}{\partial r}+X_{\upsilon}\upsilon+X_{0}$$ such that $X_{0}$ is $D-$valued. Then the auxiliary operator is 
\begin{eqnarray}\label{equ auxiliary with F}\star(F^{0}_{A}\wedge d[X\lrcorner\psi])&=& \star_{D^{\star}}\{[(L_{\frac{\partial}{\partial s}}X_{0})\lrcorner \frac{d\eta}{2}+(L_{\frac{\partial}{\partial r}}X_{0})\lrcorner H+\frac{(L_{\upsilon}X_{0})\lrcorner G}{r}]\wedge F^{0}_{A}\}\nonumber
\\&=& \star_{D^{\star}}\{[(L_{\frac{\partial}{\partial s}}X)\lrcorner \frac{d\eta}{2}+(L_{\frac{\partial}{\partial r}}X)\lrcorner H+\frac{(L_{\upsilon}X)\lrcorner G}{r}]\wedge F^{0}_{A}\}\nonumber
\\&=& -[(L_{\frac{\partial}{\partial s}}X)\lrcorner \frac{d\eta}{2}+(L_{\frac{\partial}{\partial r}}X)\lrcorner H+\frac{(L_{\upsilon}X)\lrcorner G}{r}]\lrcorner F^{0}_{A}.
\end{eqnarray}

\end{lem}

Strategy: it is completely routine.   
We simply calculate 
\begin{enumerate}
\item  the exterior derivative of $X\lrcorner \psi_{euc}$, then
\item  wedge it with the curvature then apply $\star$. 
\end{enumerate}
The idea for the first step is to separate $d(X_{0}\lrcorner \psi)$ into two parts, such that the first part contains $ds\wedge dr \wedge \eta$ as a byte, but the other does not. Then carrying out the second step,  the first part  yields the first line  on the right side of  \eqref{equ auxiliary with F}, the other part  yields the supplementary term $Q(X_{0})$ (see \eqref{equ d of Xcontractpsi}  below) which has vanishing wedge with the curvature. 

\begin{proof}[\textbf{Proof of Lemma \ref{lem formula for auxiliary operator}}:] 
The standard co-associative form on $\mathbb{C}^{3}\times \mathbb{S}^{1}$ is 
\begin{eqnarray}\label{equ Sasakian formula for psi}\psi_{euc}&=&\frac{\omega^{2}_{euc}}{2}+Im\Omega_{euc}\wedge ds \nonumber
\\&=&r^{3}dr\wedge\eta\wedge \frac{d\eta}{2}+\frac{r^{4}}{2} (\frac{d\eta}{2})^{2}-r^{3}ds\wedge\eta\wedge H+r^{2}ds\wedge dr\wedge G.
\end{eqnarray}

\subsubsection*{Step 1: The semi-basic component of the vector field} Let $X_{0}$ be a semi-basic vector field (contact distribution $D-$valued) on $\mathbb{C}^{3}\setminus O$, we compute
\begin{equation}\label{equ X0 contract psi}
X_{0}\lrcorner \psi_{euc}=r^{3}dr\wedge \eta\wedge (X_{0}\lrcorner \frac{d\eta}{2})-r^{3}ds\wedge \eta\wedge (X_{0}\lrcorner H)+r^{2}ds\wedge dr\wedge (X_{0}\lrcorner G)+\frac{r^{4}}{2}X_{0}\lrcorner (\frac{d\eta}{2})^{2}.
\end{equation}
We successively calculate the exterior derivative of each term in \eqref{equ X0 contract psi} using the Reeb Lie derivatives in  \cite[Section 3.4]{W}:
\begin{eqnarray}d[r^{3}dr\wedge \eta\wedge (X_{0}\lrcorner \frac{d\eta}{2})]&=&r^{3}ds\wedge dr\wedge \eta\wedge [\frac{\partial X_{0}}{\partial s}\lrcorner \frac{d\eta}{2}]-2r^{3}dr\wedge \frac{d\eta}{2}\wedge (X_{0}\lrcorner \frac{d\eta}{2})\\& &+r^{3}dr\wedge \eta\wedge d_{0}(X_{0}\lrcorner \frac{d\eta}{2}),\nonumber
\end{eqnarray}
\begin{eqnarray}\label{equ 1 Aux appendix}\nonumber d[r^{3}ds\wedge \eta\wedge (X_{0}\lrcorner H)]&=&-3r^{2}ds\wedge dr\wedge \eta\wedge (X_{0}\lrcorner H)-r^{3}ds\wedge dr\wedge \eta \wedge \frac{\partial}{\partial r}(X_{0}\lrcorner H)\\& &-2r^{3}ds\wedge \frac{d\eta}{2}\wedge (X_{0}\lrcorner H)+r^{3}ds\wedge \eta\wedge d_{0}(X_{0}\lrcorner H),
\end{eqnarray}
\begin{eqnarray}\label{equ 2 Aux appendix}\nonumber d[r^{2}ds\wedge dr\wedge (X_{0}\lrcorner G)]&=&r^{2}ds\wedge dr\wedge \eta\wedge [(L_{\upsilon}X_{0})\lrcorner G]-3r^{2}ds\wedge dr\wedge \eta \wedge (X_{0}\lrcorner H)\\& &+r^{2}ds\wedge dr\wedge d_{0}(X_{0}\lrcorner G).
\end{eqnarray}
Using the above $3$ identities and \eqref{equ X0 contract psi}, we find 
\begin{equation}  d(X_{0}\lrcorner \psi_{euc})\label{equ d of Xcontractpsi} 
= r^{3}ds\wedge dr\wedge \eta\wedge [\frac{\partial X_{0}}{\partial s}\lrcorner \frac{d\eta}{2}+\frac{\partial X_{0}}{\partial r}\lrcorner H+\frac{(L_{\upsilon}X_{0})\lrcorner G}{r}]+Q(X_{0}).
\end{equation}
where 
\begin{eqnarray}\label{equ Q}Q(X_{0})&=&-2r^{3}dr\wedge \frac{d\eta}{2}\wedge (X_{0}\lrcorner \frac{d\eta}{2})+\frac{r^{4}}{2}d[X_{0}\lrcorner (\frac{d\eta}{2})^{2}] +2r^{3}ds\wedge \frac{d\eta}{2}\wedge (X_{0}\lrcorner H)\\& &-r^{3}ds\wedge \eta\wedge d_{0}(X_{0}\lrcorner H) \nonumber +r^{2}ds\wedge dr\wedge d_{0}(X_{0}\lrcorner G)+r^{3}dr\wedge \eta\wedge d_{0}(X_{0}\lrcorner \frac{d\eta}{2})\\& &+2r^{3}dr\wedge [X_{0}\lrcorner(\frac{d\eta}{2})^{2}].\nonumber
\end{eqnarray}
The term $-3r^{2}ds\wedge dr\wedge \eta\wedge (X_{0}\lrcorner H)$ in \eqref{equ 1 Aux appendix} and \eqref{equ 2 Aux appendix} cancels out. 

Because the Hodge star of $ds\wedge dr\wedge \eta$ is semi basic, wedging  \eqref{equ d of Xcontractpsi} by $F^{0}_{A}$, it is to routine to verify that 
\begin{eqnarray}\label{eqn Aux with Q}Aux(X_{0})&=&\star\{F^{0}_{A}\wedge (r^{3}ds\wedge dr\wedge \eta)\wedge [\frac{\partial X_{0}}{\partial s}\lrcorner \frac{d\eta}{2}+\frac{\partial X_{0}}{\partial r}\lrcorner H+\frac{(L_{\upsilon}X_{0})\lrcorner G}{r}]\}
\\&  &+\star[F^{0}_{A}\wedge Q(X_{0})]\nonumber
\\&=&\nonumber \star_{D^{\star}}\{[(L_{\frac{\partial}{\partial s}}X_{0})\lrcorner \frac{d\eta}{2}+(L_{\frac{\partial}{\partial r}}X_{0})\lrcorner H+\frac{(L_{\upsilon}X_{0})\lrcorner G}{r}]\wedge F^{0}_{A}\}
\\&  &+\star[F^{0}_{A}\wedge Q(X_{0})].\nonumber
\end{eqnarray}

\subsubsection*{Step 2: The component of $X$ perpendicular to the contact distribution has no contribution to the auxiliary operator.}

\begin{fact}\label{fact srv do not contribute to Aux} For any $C^{1}-$ functions $X_{s},\ X_{\upsilon},\ X_{r}$  defined on a punched tubular ball in the model space, 
\begin{equation}\label{equ Aux with non basic term} Aux(X_{s}\frac{\partial}{\partial s}+X_{\upsilon}\upsilon+X_{r}\frac{\partial}{\partial r})=0.
\end{equation}
\end{fact}
The proof is completely routine.  The distribution $span(\frac{\partial}{\partial s},\frac{\partial}{\partial r},\upsilon)$ is integrable (involutive) of which $X-X_{0}$ is a section.  The observation is that the exterior differential  of each term in 
\begin{equation}\label{equ contraction from difference}(X_{s}\frac{\partial}{\partial s}+X_{\upsilon}\upsilon+X_{r}\frac{\partial}{\partial r})\lrcorner \psi_{euc} \end{equation} contains at least one among  the $3-$forms $\frac{d\eta}{2},\ G,\ H$ as a byte. This is because every term in $\psi_{euc}$ itself contains one of these as byte, and applies the identities 
$$dH=3\eta\wedge G,\ dG=-3\eta\wedge H.$$ Therefore the wedge of \eqref{equ contraction from difference} and  the anti self-dual  curvature $F^{0}_{A}$ (as an $EndE-$valued section of $\wedge^{2}D^{\star}$) vanishes.

 To complete the proof of the Lemma, it suffices to show the following which indeed requires that the co-efficients of the vector field $X$ only depend on $r,\ s$. This condition is not applied so far.  
\subsubsection*{Step 3: the wedge between each term in $Q(X_{0})$ and $F^{0}_{A}$ is $0$.} We first show it for  the 3 terms in line 2 of \eqref{equ Q}. The observation is that the multiplication by a differentiable function of only $r,\ s$ commutes with the transverse Hodge dual operator $d_{0}^{\star_{D^{\star}}}$. Namely, on the first term among the $3$, it suffices to show
$$d_{0}(X_{0}\lrcorner H)\wedge F^{0}_{A}=0.$$
Taking $\star_{D^{\star}}$, the above is equivalent to
$$-d_{0}^{\star_{D^{\star}}}[(J_{H}X_{0})\lrcorner F^{0}_{A}]=0.$$
 Using $J_{H}$ invariance of the curvature,  it suffices to observe 
$$d_{0}^{\star_{D^{\star}}}[(J_{H}X_{0})\lrcorner F^{0}_{A}]=\Sigma_{i=1}^{6}\frac{X_{i}}{r}\cdot d_{0}^{\star_{D^{\star}}}[J_{H}(re_{i}\lrcorner F^{0}_{A})]=0$$ where we used 
$$d_{0}X_{i}=0\ \textrm{for all}\ i=1,..,6$$
because these co-efficients only depend on $r$ and $s$.  The other two terms are similar. 

To show the four terms in line 1 and 3 of  \eqref{equ Q} have vanishing wedge with the curvature, 
using the identity $X_{0}\lrcorner [\frac{1}{2}(\frac{d\eta}{2})^{2}]=\star_{D^{\star}}(X_{0}^{\sharp_{D^{\star}}})$, we calculate the the second term  in line 1 of \eqref{equ Q}:
\begin{eqnarray}\label{equ 0 proof aux formula}d[X_{0}\lrcorner \frac{(d\eta)^{2}}{4}]&= &ds\wedge [\frac{\partial X_{0}}{\partial s}\lrcorner \frac{(d\eta)^{2}}{4}]+dr\wedge [\frac{\partial X_{0}}{\partial r}\lrcorner \frac{(d\eta)^{2}}{4}]+\eta\wedge [L_{\upsilon}(X_{0})\lrcorner \frac{(d\eta)^{2}}{4}]
\\& &+d_{0}[X_{0}\lrcorner \frac{(d\eta)^{2}}{4}].\nonumber
\end{eqnarray}
Any form with a byte in $\wedge^{5}D^{\star}$ must vanish because the ($\mathbb{R}$) rank of the contact distribution $D^{\star}$ is $4$. Because the curvature $F^{0}_{A}$ is an endomorphism-valued semi basic $2-$form (pullback from $\mathbb{P}^{2}$), any form with a byte in $\wedge^{3}D^{\star}$ has vanishing wedge with the curvature. This is precisely the case for every term in \eqref{equ 0 proof aux formula}. The reason why the last term is semi-basic is simply that the transverse exterior differential $d_{0}$ of a semi basic form remains semi basic.  In summary, we find 
$$\frac{r^{4}}{2}d[X_{0}\lrcorner (\frac{d\eta}{2})^{2}]\wedge F^{0}_{A}=0.$$
Similarly,  the other 3 terms in line 1 and line 3 of \eqref{equ Q} also has byte of semi basic $3-$form. Then their wedge with the curvature also vanish
$$2r^{3}dr\wedge [X_{0}\lrcorner(\frac{d\eta}{2})^{2}]\wedge F^{0}_{A}\nonumber
=-2r^{3}dr\wedge \frac{d\eta}{2}\wedge (X_{0}\lrcorner \frac{d\eta}{2})\wedge F^{0}_{A}=2r^{3}ds\wedge \frac{d\eta}{2}\wedge (X_{0}\lrcorner H)\wedge F^{0}_{A}\nonumber
=0.$$
This means  $Q(X_{0})$ has no contribution to the auxiliary operator i.e. 
$$Q(X_{0})\wedge F^{0}_{A}=0. $$
The first two equal signs in \eqref{equ auxiliary with F} is proved by \eqref{eqn Aux with Q} and \eqref{equ Aux with non basic term}. The curvature $F^{0}_{A}$ is $\star_{D^{\star}}$ anti self dual. Then $$\star_{D^{\star}}(\theta\wedge F^{0}_{A})=-\theta\lrcorner F^{0}_{A}$$ for any semi basic $1-$form $\theta$. The last line in \eqref{equ auxiliary with F} is proved. 
\end{proof}

\subsection{Homomorphism between stable bundles on $\mathbb{P}^{2}$}
In proving Lemma \ref{cor rigidity of bundle on P2}, under the Chern number condition and others  therein,  the following is crucial to bound $h^{0}[\mathbb{P}^{2},EndE]$ and to show that the poly-stable bundle $E$ is  stable. 
\begin{lem}\label{lem Homorphism} On $\mathbb{P}^{n}$, any nontrivial sheaf  homomorphism between two stable locally free sheaves of the same slope is an isomorphism. 

Consequently, the space of such homomorphisms is either (complex) $0$ or $1-$dimensional. 
\end{lem}
On projective curves, the similar result is well recorded in literature. But this particular version we need does not seem  very easy to find. Following \cite{OSS} verbatim,  we still give the detail for the reader's convenience. 
\begin{proof}Let $\phi:\ E_{1}\rightarrow E_{2}$ denote the nontrivial  homomorphism and stable bundles. \cite[Lemma 1.2.8]{OSS} says $\phi$ must be injective or generically surjective i.e. surjective on stalks at an arbitrary  point away from the singular locus of  $Coker\phi$.  By  \cite[Corollary page 171]{OSS}, it suffices to show $rankE_{1}=rankE_{2}$ by ruling out the following two cases.

Case A: suppose $rank E_{1}<rank E_{2}$. Then $\phi$ must be injective and $Image\phi$ is a sub-sheaf of $E_{2}$ of the same slope but lower rank. This  contradicts the stability of $E_{2}$.

Case B: suppose $rank E_{1}>rank E_{2}$, then it must be generically  surjective. Using that $Image\phi$ is a torsion free coherent quotient of $E_{1}$ \cite[Proof in page 170]{OSS}, we find $rankImage\phi=rankE_{2}$. Moreover, we have $c_{1}(Image \phi)\leq c_{1}(E_{2})$  \cite[Proof 1 in page 161]{OSS}. Thus  the torsion free quotient has less or equal slope: 
$$\mu(Image \phi)\leq \mu(E_{2})=\mu(E_{1}).$$
This contradicts the stability of $E_{1}$. 

The consequence holds by simpleness of stable bundles. \end{proof}

\end{document}